\documentclass{article}
\author{Neeraj Kashyap}
\title{Ranks of elliptic curves over cyclic cubic, quartic, and sextic extensions}

\usepackage{amsmath, amsthm, amssymb, tikz, enumerate}
\usetikzlibrary{calc}

\newtheorem{thm}{Theorem}
\newtheorem{absthm}{Theorem}
\newtheorem{prop}{Proposition}
\newtheorem{lemma}{Lemma}
\newtheorem{quest}{Question}
\newtheorem{defn}{Definition}

\newcommand{\dualproj}[1]{\widehat{\mathbb{P}^{#1}}}
\newcommand{\zmod}[1]{\mathbb{Z}/#1\mathbb{Z}}
\newcommand{\absgal}[1]{Gal(\overline{#1}/#1)}
\newcommand{\gal}[2]{Gal(#1/#2)}
\newcommand{\weightproj}{\mathbb{P}^{1,2,1}}
\newcommand{\ints}[1]{\mathcal{O}_{#1}}

\begin{document}

\maketitle

\begin{abstract}
For a given group $G$ and an elliptic curve $E$ defined over a number field $K$, I discuss the problem of finding $G$-extensions of $K$ over which $E$ gains rank. I prove the following theorem, extending a result of Fearnley, Kisilevsky, and Kuwata:
\begin{absthm}
Let $n = 3,4,$ or $6$. If $K$ contains its $n^{th}$-roots of unity then, for any elliptic curve $E$ over $K$, there are infinitely many $\zmod{n}$-extensions of $K$ over which $E$ gains rank.
\end{absthm}
\end{abstract}

\section*{Acknowledgments}

This paper is based on a doctoral thesis submitted to Indiana University on April 18, 2013. I am grateful to my doctoral advisor, Michael Larsen, for introducing me to this problem and for the insights he provided along the way. I am grateful to Bo-Hae Im for encouraging me to finish this paper. I found Mathematica \cite{Mathematica} useful in deriving some of the expressions within.

\section{Introduction \label{Introduction}}

Throughout this paper, $K$ denotes a number field and $E$ denotes an elliptic curve which is defined over $K$. I fix a Weierstrass equation
\begin{equation}
\label{weierstrass}
E : x^3 + Ax + B
\end{equation}
and will make frequent reference to the Weierstrass coefficients $A$ and $B$. I take the point at infinity on $E$ to be its identity, and denote it $O$.

I denote by $\ints{K}$ the ring of integers in the number field $K$. For a prime $P$ of $\ints{K}$, I denote by $k_P$ the corresponding residue field and write $q_P := |k_P|$. Finally, given a Galois extension $L$ of $K$, I denote by $\Gamma_{L/K}$ the corresponding Galois group. Throughout this paper, ``rational point'' means ``$K$-rational point''.

Every result about elliptic curves that I use without proof in this paper is at least mentioned in Silverman \cite{Silverman}.

\begin{defn}
\label{defrankgain}
I say that $E$ \emph{gains rank} over an extension $L$ of $K$ if there is a point $P \in E(L)$ such that the Galois extension generated by $nP$ is $L$ for any non-zero integer $n$. In this situation, $P$ \emph{raises the rank} of $E$ over $L$.
\end{defn}

This paper is driven by one question:
\begin{quest}
\label{rankgain}
Given a finite group $G$, what are the $G$-extensions of $K$ over which $E$ gains rank?
\end{quest}

In the next section, as a simple demonstration of my method, I will show that for any elliptic curve over any number field $K$, there are infinitely many quadratic extensions of $K$ over which $E$ gains rank. In terms of results about quadratic extensions, this is actually quite primitive. Asking whether or not $E$ gains rank over a particular quadratic extension of $K$ is the same as asking whether or not the corresponding quadratic twist of $E$ has positive rank. In this guise, the case $G = \zmod{2}$ has been well studied. Kuwata and Wang \cite{KuwataWang} and Coogan and Jimenez-Urroz \cite{twist2curves} have studied the problem of controlling rank rises over quadratic extensions of $\mathbb{Q}$ for pairs of curves. More recently, Im \cite{Im4curves} has considered the problem of finding quadratic extensions over which four curves simultaneously increase in rank. In a different direction, Dokchitser and Dokchitser \cite{DDTwist} have proved that, assuming the Birch and Swinnerton-Dyer conjecture, there are elliptic curves over number fields $K \neq \mathbb{Q}$ \emph{all} of whose quadratic twists are of positive rank. Vatsal \cite{VTwist} has showed that there is an elliptic curve over $\mathbb{Q}$ such that a positive density of its twists are of positive rank. Im and Larsen \cite{ILPP} have observed that, among the twists of Vatsal's curve, one can find elliptic curves which have positive rank twists at $\mathbb{F}_2$-subspaces of arbitrary dimension. In a \emph{third} direction, there has been quite a bit of work on estimating how many twists of an elliptic curve have rank greater than 1. Silverberg \cite{SilverbergTwist} is a good survey of such results.

The case $G = \zmod{3}$ has been studied by Fearnley, Kisilevsky, and Kuwata in \cite{FKK}. The authors prove that if $K$ contains its cube roots of unity then there are infinitely many cyclic cubic extensions of $K$ over which $E$ gains rank.\footnote{This is the case $n = 3$ of Theorem \ref{main}.} They also show that if an elliptic curve $E$ over $\mathbb{Q}$ has at least six rational points then there are infinitely many cyclic cubic extensions of $\mathbb{Q}$ over which it gains rank.

David, Fearnley, and Kisilevsky in \cite{KConjecture} predict that there should also be infinitely many $\zmod{5}$-extensions of $\mathbb{Q}$ over which an elliptic curve $E$ defined over $\mathbb{Q}$ gains rank, but that the same should not be true for $\zmod{p}$-extensions for primes $p > 5$.

In this paper, I present a new method of answering such questions. Using this method, I prove the following theorem which builds upon the first result of Fearnley, Kisilevsky, and Kuwata \cite{FKK}:
\begin{thm}
\label{main}
Let $n = 3,4,$ or $6$. If $K$ contains its $n^{th}$-roots of unity then, for any elliptic curve $E$ over $K$, there are infinitely many $\zmod{n}$-extensions of $K$ over which $E$ gains rank.
\end{thm}

The method itself is simple. Suppose, as above, that one wants to find $G$-extensions of $K$ over which $E$ gains rank. There are three steps:
\begin{enumerate}[I.]
\item{ \label{action}\emph{Action -- Find a product $E^n$ upon which $G$ acts algebraically and faithfully, with the action defined over $K$.}}
\item{ \label{rationality}\emph{Rationality -- Find $K$-rational points on the variety $E^n/G$.}\\
Since the action of $G$ on $E$ is defined over $K$, $K$-rational points on $E^n/G$ parametrize $G$-orbits in $E$ on which the $G$ action may be induced by the action of the absolute Galois group $\absgal{K}$.}
\item{ \label{irreducibility}\emph{Irreducibility -- From the list of points produced in step 2, find those which are \emph{irreducible}.}\\
A $G$-orbit in $E^n$ which is defined over $K$ may split into more than one orbit under the action of $\absgal{K}$ or the $G$-action upon that orbit may not be free. In this case the corresponding points on $E$ are defined over smaller extensions. An orbit of defines a $G$-extension if and only if this is not the case. This motivates the following definition and provides a proof of the following proposition:
\begin{defn}
\label{defirreducible}
A $K$-rational point on $E^n/G$ is said to be \emph{irreducible} if the action of $\absgal{K}$ on the corresponding $G$-orbit in $E^n$ is transitive.
\end{defn}

\begin{prop}
\label{irredrankgain}
Let $x \in E^n/G$ be $K$-rational, and let $S$ denote the corresponding $G$-orbit of points in $E^n$. Then $K(S)$ is a $G$-extension of $K$ if and only if $x$ is irreducible.
\end{prop}

None of the discussion up to this point has any direct bearing on rank increase. The strategy to show that rank is actually gained over some $G$-extension is to produce infinitely many irreducible points on $E^n/G$. By Merel's theorem \cite{Merel}, only finitely many of the corresponding orbits can be $G$-orbits of torsion points. Therefore, there is at least one point of positive rank on $E$ which defined only over a $G$-extension of $K$. This still leaves the question of whether this point simply corresponds to a point of positive rank in $E(K)$. To show that this is never the case for the points I produce in this paper, I make use of the following simple trick:
\begin{lemma}
\label{mereltrick}
Suppose $L/K$ is Galois and $P$ is a point on $E$ defined only over $L$ such that
\[\sum_{\sigma \in \gal{L}{K}}P^{\sigma} = O.\]
If $n\cdot P \in E(K)$ for some positive integer $n$, then $P$ is a torsion point.
\end{lemma}

\begin{proof}
Put $P' := n\cdot P \in E(K)$. Then $P' = n\cdot P^{\sigma}$ for all $\sigma \in \gal{L}{K}$. Therefore,
\[O = n\sum_{\sigma \in \gal{L}{K}}P^{\sigma} = |\gal{L}{K}|\cdot P',\]
and so $P$ is a torsion point.\qedhere
\end{proof}

All the rational points I produce from step \ref{rationality} in this paper will satisfy the hypothesis of Lemma \ref{mereltrick}. This guarantees that $E$ truly gains rank over some $G$-extension of $K$.

To prove that there are infinitely many such $G$-extensions will require careful analyses of my sources of rational points on $E^n/G$, but the basic idea is that these sources are parametrized and using these parametrizations one produce from the irreducible points $G$-extensions with arbitrary ramification behaviour over various large primes, thereby proving in particular that there are infinitely of them.
}
\end{enumerate}

In the following section, I demonstrate this method in the case of quadratic extensions. In section 3, I describe actions of the groups $\zmod{3}, \zmod{4},$ and $\zmod{6}$ on the surface $E^2$. In section 4, I carry out the ``Rationality'' step of the procedure by producing sources of rational points on the varieties $E^2/(\zmod{n})$ for $n = 3,4,6$ under the assumptions of Theorem \ref{main}. In section 5, I carry out the ``Irreducibility'' step and complete my proof of Theorem \ref{main}.

\section{The (motivating) case of quadratic extensions \label{QuadraticCase}}

In this section, I prove the following result:
\begin{prop}
\label{quadratic}
For any elliptic curve $E$ defined over any number field $K$, there are infinitely many quadratic extensions of $K$ over which $E$ gains rank.
\end{prop}

This can also be stated in the following (perhaps more familiar) form:
\begin{prop}
For any elliptic curve $E$ defined over any number field $K$, there are infinitely many quadratic twists of $E$ over $K$ of positive rank.
\end{prop}

\begin{proof}
I will break up my proof into the steps of the method presented in the introduction.

\begin{enumerate}[I.]
\item{ \emph{(Action)} The group $\zmod{2} = \langle s \rangle$ acts naturally upon $E$ itself by the inversion $s\cdot P = -P$.}
\item{ \emph{(Rationality)} The quotient of $E$ under this action of $\zmod{2}$ is isomorphic to $\mathbb{P}^1$. This isomorphism is given by the Weierstrauss $x$-coordinate. Thus one may find as many $K$-rational points as one likes on $E/\left(\zmod{2}\right)$ -- the equivalence classes $\{P, -P\}$ for $P \in E$ such that $x(P) \in K$ or to the equivalence class $\{O\}$.}
\item{ \emph{(Irreducibility)} A point $P$ on $E$ is defined over a quadratic extension of $K$ if and only if the equivalence class $\{P,-P\} \in E/\left(\zmod{2}\right)$ is $K$-rational, and $P$ can only be defined over a quadratic extension of $K$ if $\{P,-P\}$ forms a single orbit under the action of $\absgal{K}$ on $E$.

For a point $P$ on $E$ such that $\lambda := x(P) \in K$, $\{P, -P\}$ is a single Galois orbit if and only if the polynomial $p_{\lambda}(y) = y^2 - \left(\lambda^3 + A\lambda + B\right)$ is irreducible over $K$. If this is the case, then the quadratic extension of $K$ over which $P$ is defined is generated by the roots of $p_{\lambda}(y)$.

Let $P$ be a prime of $K$, and consider the polynomials $p_{\lambda}(y)$ with $\lambda, y \in k_P$. By Hasse's bound, for ideals $P$ of large norm, the number of pairs $(\lambda, y)\in k_P\times k_P$ such that $p_{\lambda}(y) = 0$ is roughly $q_P$. Since $p_{\lambda}(y) = p_{\lambda}(-y)$, there are roughly $q_P/2$ values of $\lambda$ in $k_P$ for which $p_{\lambda}(y)$ factors over $k_P$. Therefore, given any two disjoint lists of large primes $P_1, P_2, \ldots, P_n$ and $Q_1, Q_2, \ldots, Q_m$ in $\ints{K}$, there exists $\lambda \in K$ such that $p_{\lambda}(y)$ is irreducible modulo each of the ideals $P_i$ but factors modulo each of the ideals $Q_j$. This means that there are points on $E$ defined over quadratic extensions of $K$ which are ramified however one likes over large primes. Thus there are infinitely many quadratic extensions $L$ of  $K$ such that there are points $P$ on $E$ defined only over $L$.

If $P$ is a point produced by the above procedure, it $P$ and $-P$ are Galois conjugates, and so $P$ satisfies the hypothesis of Lemma \ref{mereltrick}. Therefore, by the lemma and Merel's theorem, there are infinitely many quadratic extensions of $K$ over which $E$ gains rank.\qedhere
}
\end{enumerate}
\end{proof}

My proof of Theorem \ref{main} will follow this rubric.

\section{Actions \label{Actions}}

The question of whether or not a group $G$ acts faithfully on the variety $E^n$ is equivalent to that of whether or not $G$ has a faithful $n$-dimensional representation over the endomorphism ring of $E$. Generically, the question is whether or not $G$ embeds in $GL_n(\mathbb{Z})$.

Step \ref{action} of the procedure I described in the introduction requires one to find an integer $n$ such that $G$ acts faithfully on $E^n$. The obvious approach to the problem of finding such an $n$ is to note that if $n = |G|$ then $G$ embeds in $S_n$ which acts by its permutation representation on $E^{n-1}$. In order to carry out the procedure to completion, it is desirable that the dimension of the representation in consideration be as small as possible. This raises the following question:
\begin{quest}
\label{minact}
Given a finite group $G$ and an elliptic curve $E$, what is the smallest integer $n$ such that $G$ acts faithfully and algebraically on $E^n$?
\end{quest}

I describe a generalization of the observation regarding the symmetric groups above which leads to the actions that I consider for the groups $\zmod{3}, \zmod{4},$ and $\zmod{6}$, and which could be useful in answering other instances of Question \ref{rankgain}. The minimality of these three actions in the sense of Question \ref{minact} above will be apparent.

Let $E$ be an elliptic curve, $R$ a root system of rank $r$, $\Lambda$ the lattice generated by the dual to $R$, and $W$ the Weyl group of $R$. Then $W$ acts on $E^r$ via the isomorphism $E^r \cong \Lambda \otimes E$. This situation was studied by Looijenga \cite{Looijenga}, who proved that the quotients $E^r/W$ arising from the above actions are in fact weighted projective spaces (and therefore have plenty of rational points).

The irreducible rank 2 root systems $A_2$, $B_2$, and $G_2$ have Weyl groups $D_3$, $D_4$, and $D_6$ respectively. The embeddings of $\zmod{3}$, $\zmod{4}$, and $\zmod{6}$ into these dihedral groups give rise to the actions I will deal with. The rest of this section consists of explicit descriptions of the actions of the dihedral groups.

For a dihedral group $D_n$, I will use the presentation
\[D_n = \langle r_n, s_n : r_n^n = s_n^2 = 1, s_nr_ns_n = r_n^{-1}\rangle.\]
When the order of the dihedral group is clear from context, I will drop the subscripts on $r_n$ and $s_n$.

The action of $D_3$ on $E^2$ arises from the permutation representation of $S_3$ since $D_3 \cong S_3$ and
\[E^2 \cong K = \{(P,Q,R) \in E^3 : P + Q + R = O\}.\]
Explicitly,
\begin{equation}
\label{D3act}
\begin{array}{rcl}
r\cdot (P,Q) & = & (Q,-P-Q),\\
s\cdot (P,Q) & = & (P,-P-Q).
\end{array}
\end{equation}

It is natural now to present the action of $D_6$ on $E^2$, as it must restrict on the subgroup $\langle r_6^2, s_6 \rangle \cong D_3$ to the $D_3$-action above. Moreover, the action of the involution $r_6^3$ must commute with that of $s_6$, and from (\ref{D3act}) one has
\[s_6\cdot (P,Q) = (P, -P-Q).\]
This forces
\begin{equation}
\label{secondref}
r_6^3\cdot (P,Q) = (-P,-Q).
\end{equation}
The actions of $r_6^2$ and $r_6^3$ on $E^2$ yield
\begin{equation}
\label{D6act}
\begin{array}{rcl}
r_6\cdot (P,Q) & = & (P+Q,-P),\\
s_6\cdot (P,Q) & = & (P,-P-Q).
\end{array}
\end{equation}

Finally, the $D_4$-action on $E^2$ is given by taking the tensor product of the natural 2-dimensional representation of $D_4$ on $\mathbb{Z}^2$ with $E$. Explicitly,
\begin{equation}
\label{D4act}
\begin{array}{rcl}
r\cdot (P,Q) & = & (-Q,P),\\
s\cdot (P,Q) & = & (P,-Q).
\end{array}
\end{equation}

\section{Rational points \label{RationalPoints}}

As mentioned in the previous section, one has from Looijenga \cite{Looijenga} that the quotients $E^2/D_3, E^2/D_4$, and $E^2/D_6$ by the above actions are weighted projective spaces. These quotients admit the varieties $E^2/(\zmod{3}), E^2/(\zmod{4}),$ and $E^2/(\zmod{6})$ respectively as double covers. As weighted projective varieties, the quotients of $E^2$ by the dihedral groups have an abundance of rational points. In this section, I will produce certain rational curves on these dihedral quotients which contain infinitely many rational points that lift to rational points in the corresponding double cover, none of which lie on the locus of points over which the quotient map from $E^2$ is ramified. Before I can do so, I must first give a better description of the quotients of $E^2$ by the groups $D_3$, $D_4$, and $D_6$.

\subsection{Dihedral quotients \label{RationalPointsDihedralQuotients}}

I begin with the quotients $E^2/D_3$ and $E^2/D_6$. Note that each $D_3$-orbit of $E^2$ consists simply of all possible permutations of three collinear points on $E$. Over $\mathbb{C}$, these sets of collinear points are parametrized by the lines in $\mathbb{P}^2$, i.e. the points in the dual projective plane $\dualproj{2}$. Therefore,
\begin{equation}
\label{D3quotient}
E^2/D_3 \cong \dualproj{2}.
\end{equation}

To understand the $D_6$-action on $E^2$, observe that the $D_3$-action on $E^2$ is induced by the $D_6$-action by taking as generators of $D_3$ the elements $r_3 = r_6^2$ and $s_3 = s_6$ of $D_6$. This allows one to consider $E^2/D_3 \cong \dualproj{2}$ as a double cover of $E^2/D_6$ and therefore $E^2/D_6$ as the quotient of $\dualproj{2}$ by the action of $r_6^3$. Impose coordinates $X, Y, Z$ on $\mathbb{P}^2$ such that $X/Z$ and $Y/Z$ respectively give the Weierstrass $x$- and $y$-coordinates of (\ref{weierstrass}) in the affine plane $Z \neq 0$. The point $[a:b:c] \in \dualproj{2}$ corresponds to the line
\[aX + bY + cZ = 0\]
in $\mathbb{P}^2$. In these coordinates, it follows from (\ref{secondref}) that the action of $r_6^3$ on $\dualproj{2}$ is given by
\[r_6^3 \cdot [a:b:c] = [a:-b:c].\]

Thus $E^2/D_6$ can be constructed from $\dualproj{2}$ by identifying the points $[a:b:c]$ and $[a:-b:c]$. The result of this identification is the \emph{weighted projective space} 
\begin{equation}
\label{D6quotient}
\weightproj \cong E^2/D_6.
\end{equation}

Explicitly, over a field $F$, $\weightproj(F)$ is the set of points of the form $[a':b':c']$ with the identification
\[[a':b':c'] = [\lambda a':\lambda^2b':\lambda c'],\ \lambda \in F^*.\]
The covering map $\psi: \dualproj{2} \rightarrow \weightproj$ is given by
\begin{equation}
\label{D3D6cover}
\psi([a:b:c]) = [a:b^2:c].
\end{equation}

Understanding the geometry of $E^2/D_4$ takes some more effort. The $D_4$-orbit of a point $(P,Q) \in E^2$ is
\[\{(P,Q), (Q,-P), (-P,-Q), (Q,-P), (P,-Q), (-Q,-P), (-P,Q), (Q,P)\}.\]
Thus the $D_4$-orbit of $(P,Q)$ specifies $P$ and $Q$ up to signs and up to their order.

A point on $E$ is specified up to sign by its Weierstrass $x$-coordinate, and the quotient by the identification of points on $E$ by their $x$-coordinates is $E/(\zmod{2}) \cong \mathbb{P}^1$. Consider the Klein 4-group $V = \langle r^2, s \rangle$. Since
\[V\cdot (P,Q) = \{(P,Q), (-P,-Q), (P,-Q), (-P,Q)\},\]
one has
\begin{equation}
\label{Vquotient}
E^2/V \cong \mathbb{P}^1 \times \mathbb{P}^1.
\end{equation}

Taking a further quotient by the remaining action of $r$ corresponds to forgetting about order, and thus yielding the symmetric square of $\mathbb{P}^1$ which is $\mathbb{P}^2$. This gives an isomorphism
\begin{equation}
\label{D4quotient}
E^2/D_4 \cong \mathbb{P}^2.
\end{equation}

The quotient map $\phi: E^2 \rightarrow \mathbb{P}^2$ from (\ref{D4quotient}) is defined at a point $(P,Q) \in E^2$ with $P \neq O \neq Q$ by
\begin{equation}
\label{D4quotientmap}
\phi(P,Q) = [x(P) + x(Q) : x(P)x(Q) : 1].
\end{equation}

\subsection{Rational points on $E^2/(\zmod{3})$ \label{RationalPointsCubicCase}}

I begin by giving a description of the preimage in $E^2/(\zmod{3})$ of a point $[a:b:c] \in \dualproj{2} \equiv E^2/D_3$. As observed in the previous section, the point $[a:b:c] \in \dualproj{2}$ corresponds as a $D_3$-orbit to all possible ordered pairs in $E^2$ chosen from the three points of intersection of the line
\[l : aX + bY + cZ = 0\]
in $\mathbb{P}^2$ with the curve $E$ in its Weierstrass coordinates. To calculate this interesection, in the affine plane $Z \neq 0$, apply the substitution $y = -\frac{ax + c}{b}$ in the Weierstrass equation $y^2 = x^3 + Ax + B$, which produces the equation
\begin{equation}
\label{pl}
p_l(x) := x^3 - \left(\frac{ax + c}{b}\right)^2 + Ax + B = 0.
\end{equation}

The roots of $p_l(x)$ determine the $x$-coordinates of the points of intersection of $l$ with $E$, and the line $l$ itself then determines the points exactly. There are two distinct orderings of these points up to cyclic permutation, and one can distinguish between these orderings using the square root of the discriminant $\Delta_l$ of the polynomial $p_l(x)$, which is a polynomial function $\delta(P,Q)$ on $E^2$ in the Weierstrass coordinates of $P$ and $Q$. This is because, if $P + Q + R = O$ on $E$, then $x(R)$ is determined as a polynomial function of the Weierstrass coordinates of $P$ and $Q$, and
\[\delta(P,Q) = \left(x(P)-x(Q)\right)\left(x(P)-x(R)\right)\left(x(Q)-x(R)\right),\]
with
\[\delta(P,Q)^2 = \Delta_l.\]
Therefore, the quotient map $E^2 \rightarrow E^2/(\zmod{3})$ is given by
\begin{equation}
\label{Z3param}
(P,Q) \mapsto \left(\hat{l}, \delta(P,Q)\right),
\end{equation}
where $l$ is the line joining $P$ and $Q$ in $\mathbb{P}^2$ and $\hat{l}$ is the corresponding point in $\dualproj{2}$.

Observe that the covering map $E^2/\zmod{3} \rightarrow E^2/D_3$ is ramified over the curve $\Delta_l = 0$ in $\dualproj{2}$. This is a curve of degree 6 in $\dualproj{2}$ which has 9 cusps -- one at each of the points in $\dualproj{2}$ which correspond to the lines in $\mathbb{P}^2$ tangent to $E$ at its 3-torsion points.

Note further that, from (\ref{Z3param}), a rational point $[a:b:c] \in \dualproj{2}$ lifts to rational points in $E^2/(\zmod{3})$ if and only if the corresponding discriminant $\Delta_l$ is a square in $K^*$.

I will produce rational points on $E^2/(\zmod{3})$ by lifting a particular curve in $E^2/D_3$ which is rational if $K$ contains its cube roots of unity. This is not the only way of constructing points on the double cover, but it is the one that I focus on in this paper. For a different approach, refer to the paper of Fearnley, Kisilevsky, and Kuwata \cite{FKK}.

As an example of the lifting-of-curves approach, suppose that $E$ has a non-trivial 3-torsion point $T$ defined over $K$. Consider the line in $\dualproj{2}$ joining the points corresponding to the lines in $\mathbb{P}^2$ which lie tangent to $E$ at $T$ and $-T$. This is simply the line in $\dualproj{2}$ corresponding to the point of intersection in $\mathbb{P}^2$ of the lines tangent to $E$ at $T$ and $-T$. Denote by $\lambda$ this line in $\dualproj{2}$. $\lambda$ lifts to a conic in $E^2/(\zmod{3})$ and the question is whether or not this conic has a rational point. The point in $\dualproj{2}$ corresponding to the $x$-axis in $\mathbb{P}^2$ (which passes through the non-trivial 2-torsion points of $E$) is incident to $\lambda$. If the 2-torsion points of $E$ are defined over $K$, then this point lifts to rational points in $E^2/(\zmod{3})$, which are both incident to the conic in question. Therefore, if $E(K)$ has a torsion group with a subgroup isomorphic to $\zmod{2} \times \zmod{6}$, one can produce many rational points on $E^2/(\zmod{3})$. Using the same techniques I will present in Section 5 of this paper, one can show that infinitely many of these rational points are irreducible, and that they in fact produce infinitely many $\zmod{3}$-extensions of $K$ over which $E$ gains rank. I prefer in this paper to avoid imposing any conditions on the torsion of $E$ over $K$, and so I have narrowed the focus to Theorem \ref{main}.

By Hurwitz's formula, a curve $\mathcal{Q}$ in $\dualproj{2}$, if it has genus 0, lifts to a genus 0 curve in the desingularization of $E^2/(\zmod{3})$ if and only only if it intersects the curve $\mathcal{R}: \Delta_l  = 0$ (over which the quotient map $E^2/(\zmod{3}) \rightarrow E^2/D_3$ is ramified) at exactly two points other than its cusps. Generically, the action of $\absgal{K}$ on $E[3]$ is as $GL_2(\mathbb{F}_3)$. Therefore, in order for $\mathcal{Q}$ to be rational, it must intersect $\mathcal{R}$ with equal multiplicity at each of the cusps corresponding to \emph{non-trivial} 3-torsion points on $E$. Finally, if $\mathcal{Q}$ is of degree $d$, one can ensure that it is rational by requiring it to have a degree $d-1$ point at the point $[0:0:1] \in \dualproj{2}$ which is the location of the cusp on $\mathcal{R}$ corresponding to the identity $O \in E$. By Bezout's theorem, then, the degree $d$ of $\mathcal{Q}$ must satisfy $d = 4m$ where $m$ is the multiplicity of the intersection of $\mathcal{Q}$ with each of the cusps on $\mathcal{R}$ corresponding to the non-trivial 3-torsion points on $E$. Moreover, the space of curves of degree $d$ satisfying these conditions has dimension 0. If such a curve exists, then it is unique.
\begin{lemma}
\label{quartic}
For each elliptic curve $E/K$, the curve
\begin{equation}
\label{thequartic}
\mathcal{Q}_E : a^4 - 3 A b^4 + 6 a b^2 c = 0
\end{equation}
is the unique curve in $\dualproj{2}$ passing through the points corresponding to the lines of inflection of $E$ at its non-trivial 3-torsion points with equal multiplicities, and with a multiplicity $deg(\mathcal{Q}_E) - 1$ point at $[0:0:1] \in \dualproj{2}$, which corresponds to the line at infinity in $\mathbb{P}^2$.
\end{lemma}

\begin{proof}
The uniqueness was discussed in the paragraph leading up to the statement of the lemma, so I need only establish that the curve $\mathcal{Q}_E$ defined above satisfies the stated conditions. That $\mathcal{Q}_E$ is incident to $[0:0:1]$ with multiplicity 3 is clear. Therefore, all that is left is to verify the claim about the lines of inflection.

Consider the $3^{rd}$ division polynomial $\psi_3(x)$ of $E$, which vanishes precisely at the Weierstrass $x$-coordinates of the non-trivial 3-torsion points on $E$. Explicitly,
\[\psi_3(x) = 3x^4 + 6Ax^2 + 12Bx - A^2.\]
The tangent line to $E$ at a point $(u,v)$ in the $xy$-plane has the equation
\[(3u^2 + A)(x - u) - 2v(y - v) = 0.\]
This corresponds to the point $[3u^2 + A: -2v: -3u^3 - Au + 2v^2] \in \dualproj{2}$. At this point, (\ref{thequartic}) becomes
\[(3u^2 + A)^4 - 48Av^4 + 24(A + 3u^2)v^2(2v^2 - 3u^3 - Au).\]
Upon applying the substitution $v^2 = u^3 + A u + B$, this becomes the expression
\[\psi_3(u)^2,\]
which is zero precisely when $u = x(T)$ for $T$ a non-trivial 3-torsion point on $E$. Thus $\mathcal{Q}_E$ is incident to the points in $\dualproj{2}$ corresponding to the lines of inflection of $E$. Since the coefficients in (\ref{thequartic}) are rational over the function field of $E$, it must be incident to the lines of inflection at the non-trivial 3-torsion points to equal orders.
\end{proof}

As $\mathcal{Q}_E$ is a degree 4 curve with a point of multiplicity 3, it has a rational parametrization. The point on $\mathcal{Q}_E$ corresponding to a particular value of the parameter $\lambda$ is given by
\begin{equation}
\label{quarticparam}
\hat{l}_{\lambda} := [6\lambda^2:6\lambda^3:3A\lambda^4 - 1] \in \dualproj{2},
\end{equation}
with $l_{\lambda}$ denoting the corresponding line in $\mathbb{P}^2$. Write
\[p_{\lambda}(x) := p_{l_{\lambda}}(x),\]
where $p_{l_{\lambda}}(x)$ is given by (\ref{pl}) and write $\Delta_{\lambda}$ for its discriminant. The point $\hat{l}_{\lambda}$ lifts to rational points in $E^2/(\zmod{3})$ if and only if $\Delta_{\lambda} \in (K^*)^2$.
\begin{lemma}
\label{quarticdisc}
The discriminant $\Delta_{\lambda}$ is given by the formula
\[\Delta_{\lambda} = -3888(27A^2\lambda^8 - 108B\lambda^6 - 18A\lambda^4  - 1)^2.\]
\end{lemma}

\begin{proof}
Writing the equation for $l_{\lambda}$ as
\[6\lambda^3y = -(6\lambda^2x + 3A\lambda^4 - 1),\]
squaring both sides, making the substitution $y^2 = x^3 + Ax + B$, and then calculating the discriminant of the resulting polynomial equation produces the formula.
\end{proof}

This gives a source of rational points on $E^2/(\zmod{3})$ when $K$ contains its cube roots of unity.
\begin{prop}
\label{Z3rational}
If $K$ contains $\sqrt{-3}$, then $\Delta_{\lambda}$ is a square in $K$ for every $\lambda \in K^*$. In this case, every rational point on $\mathcal{Q}_E$ lifts to rational points on $E^2/(\zmod{3})$.
\end{prop}

\begin{proof}
To see this, simply factor $-3888 = -2^4\cdot 3^5$. Thus, by Lemma \ref{quarticdisc}, $\Delta_{\lambda}$ is a square in $K$ if and only if $-3$ is a square.
\end{proof}

\subsection{Rational points on $E^2/(\zmod{6})$ \label{RationalPointsSexticCase}}

The geometry of the quotient $E^2/D_6$ is closely related to that of $E^2/D_3$ via covering map $\psi$ of (\ref{D3D6cover}) in Section \ref{RationalPointsDihedralQuotients}. The focus in this section is the quotient map $\eta: E^2/(\zmod{6}) \rightarrow E^2/D_6$ associated with the action of the reflection $s_6 \in D_6$. The first objective is to establish a criterion for the rationality of lifts of rational points in $E^2/D_6$ under $\eta$. This will be achieved by considering the lifts under $\psi$, and I now introduce notation to describe these lifts. For
\[\alpha = [a' : b' : c'] \in \weightproj \cong E^2/D_6,\]
denote
\[\alpha_+ := [a' : \sqrt{b'} : c'], \alpha_- = [a' : -\sqrt{b'} : c'],\]
so that $\psi^{-1}(\alpha) = \{\alpha_+, \alpha_-\}$. Denote by $l_+$ and $l_-$ the lines in $\mathbb{P}^2$ corresponding to $\alpha_+$ and $\alpha_-$.

A rational point $\alpha \in \weightproj$ lifts to rational points on its double cover $E^2/(\zmod{6})$ if and only if the geometric action of $s_6$ on the lifts $\eta^{-1}(\alpha)$ is not induced by the action of $\absgal{K}$. As the action of $s_6$ on $E^2$ commutes with that of $r_6^3$, this is equivalent to requiring that the action of $s_6 = s_3$ on the lifts to the quotient $E^2/(\zmod{3})$ of the lifts $\psi^{-1}(\alpha)$ not be induced by the Galois action. As observed in the previous section, this latter condition holds if and only if the discriminant $\Delta_{l_+} = \Delta_{l_-}$ is a square in $K$.

Consider the image of the curve $\mathcal{Q}_E$ of (\ref{thequartic}) in $E^2/D_6$. The rational points on the image are the images of the points $[a:b:c]$ on $\mathcal{Q}_E$ which satisfy, up to the action of scalars, the conditions
\[a, c \in K, b^2 \in K.\]
When $b$ itself is rational, the geometric action of $r_6^3$ on the set $\{[a:b:c], [a:-b:c]\}$ can not be induced by $\absgal{K}$. The situation is different when $b$ is quadratic over $K$. Using the parametrization (\ref{quarticparam}), one has $\lambda^2 \in K$ if and only if $b^2 \in K$ and that $\lambda$ is quadratic over $K$ if and only if $b$ is quadratic over $K$.

\begin{prop}
\label{Z6rational}
If $K$ contains $\sqrt{-3}$, then $\Delta_{\lambda}$ is a square for every $\lambda = \pm\sqrt{\omega}$ where $\omega \in K^*\backslash\left(K^*\right)^2$. The lifts to $E^2/(\zmod{6})$ of the images in $E^2/D_6$ of such points $\hat{l}_{\lambda}$ are rational.
\end{prop}

\begin{proof}
From (\ref{quarticdisc}), if $\lambda = \pm\sqrt{\omega}$, then
\[
\Delta_{\lambda} = -3888(27A^2\omega^4 - 108B\omega^3 - 18A\omega^2 - 1)^2.
\]
Observe that $-3888 = \left(2^2\cdot 3^2\cdot \sqrt{-3}\right)^2$ as in Proposition \ref{Z3rational}.
\end{proof}

\subsection{Rational points on $E^2/(\zmod{4})$ \label{RationalPointsQuarticCase}}

Denote by $[U:V:W]$ the homogeneous coordinates on $E^2/D_4 \equiv \mathbb{P}^2$ and write $u := U/W$ and $v := V/W$. The surface $E^2/(\zmod{4})$ can be coordinatized by $U, V,$ and $W$ along with a function on $E^2$ which is $\zmod{4}$-invariant but not $D_4$-invariant under the action (\ref{D4act}).
\begin{lemma}
\label{cycinv}
The function $\omega$ defined for $(P,Q) \in E^2$ with $P \neq O \neq Q$ by
\[\omega(P,Q) := y(P)y(Q)(x(P) - x(Q))\]
is $\zmod{4}$-invariant but is not $D_4$-invariant.
\end{lemma}

\begin{proof}
By (\ref{D4act}),
\[\omega(r\cdot (P,Q)) = \omega(-Q, P) = y(-Q)y(P)(x(-Q) - x(P)).\]
The function $x$ on $E$ is even whereas $y$ is odd, so
\[\omega(r\cdot (P,Q)) = \omega(P,Q).\]
Hence $\omega$ is $\zmod{4}$-invariant. Similarly,
\[\omega(s\cdot (P,Q)) = \omega(P,-Q) = -\omega(P,Q).\]
Hence $\omega$ is not $D_4$-invariant.
\end{proof}

Note that $\omega$ extends to the whole of $E^2$ by
\[\omega(P,Q) = Y(P)Y(Q)(X(P)Z(Q) - X(Q)Z(P)).\]
The double cover $E^2/(\zmod{4})\ \rightarrow E^2/D_4$ is ramified over the $D_4$-orbits of points $(P,Q)$ such that $P$ or $Q$ is a 2-torsion point on $E$ or $P = \pm Q$. The locus of points satisfying the first condition consists of four lines in $E^2/D_4 \equiv \mathbb{P}^2$ and the locus of points satisfying the second condition is a conic. The functions $y(P)^2$ and $y(Q)^2$ can be expressed in terms of the functions $x(P) + x(Q)$ and $x(P)x(Q)$ through the Weierstrass equation (\ref{weierstrass}). In $uv$-coordinates,
\[y(P)^2y(Q)^2 = \psi(u,v),\]
where
\[\psi(u,v) := Bu^3 + v^3 + Au^2v - 3Buv - 2Av^2 + ABu + A^2v + B^2.\]
Similarly,
\[(x(P) - x(Q))^2 = u^2 - 4v.\]
Thus, for $(P,Q) \in \phi^{-1}([u:v:1])$,
\begin{equation}
\label{omegasquared}
\omega(P,Q)^2 = \psi(u,v)(u^2 - 4v).
\end{equation}
In particular, in the $uv$-plane,
\[\mathcal{R} : \psi(u,v)(u^2 - 4v) = 0.\]

Writing $P_1, P_2, P_3$ for the non-trivial 2-torsion points of $E$, the singular points on the ramification locus are the points
\[\phi(O,O), \phi(O, P_i), \phi(P_i, P_j) \in E^2,\]
for $1 \leq i,j \leq 3$ and $\phi$ as in (\ref{D4quotientmap}). Generically, these points form four Galois orbits -- the point $\phi(O,O)$, the points $\phi(O, P_i)$, the points $\phi(P_i, P_i)$, and the points $\phi(P_i, P_j)$ for $i \neq j$. By the same considerations involving Hurwitz's formula that I discussed before Lemma \ref{quartic}, there is a cubic curve in $E^2/D_4 \equiv \mathbb{P}^2$ the lift of which to $E^2/(\zmod{4})$ has genus 0.
\begin{lemma}
\label{cubic}
For an elliptic curve $E/K$, the curve
\begin{equation}
\label{thecubic}
\mathcal{C}_E : U^3 - 3UVW + AUW^2 + 2BW^3 = 0
\end{equation}
is the unique cubic curve with a double point at $\phi(O,O)$ passing through each of the points $\phi(P,Q)$ where $P$ and $Q$ are non-trivial 2-torsion points on $E$.
\end{lemma}

\begin{proof}
I show that $\mathcal{C}_E$ exhibits the stated incidence properties. The uniqueness then follows from a dimension counting argument similar to that of Lemma \ref{quartic}, which I omit.

Extending (\ref{D4quotientmap}) to all of $E^2$ gives $\phi(O,O) = [0:1:0]$, which is a double point of $\mathcal{C}_E$.

By (\ref{D4quotientmap}), for $1 \leq i \leq 3$,
\[\phi(P_i, P_i) = [2\alpha_i : \alpha_i^2 : 1].\]
Substituting into (\ref{thecubic}) yields
\begin{equation*}
\begin{array}{ccl}
\left(2\alpha_i\right)^3 - 3\left(2\alpha_i\right)\left(\alpha_i^2\right) + A\left(2\alpha_i\right) + 2B & = & 8\alpha_i^3 - 6\alpha_i^3 + 2A\alpha_i + 2B\\
 & = & 2\left(\alpha_i^3 + A\alpha_i + B\right)\\
 & = & 0
\end{array},
\end{equation*}
proving $\mathcal{C}_E$ is incident to $\phi(P_i, P_i)$.

For $1 \leq i, j \leq 3$ with $i \neq j$, choose $k$ so that $\{i,j,k\} = \{1, 2, 3\}$. If $\alpha_k = 0$, then $\phi(P_i, P_j) = [0 : A : 1]$, which is clearly incident to $\mathcal{C}_E$. Otherwise,
\[\phi(P_i, P_j) = \left[-\alpha_k : -\frac{B}{\alpha_k} : 1\right].\]
Substituting into (\ref{thecubic}) gives
\begin{equation*}
\begin{array}{ccl}
\left(-\alpha_k\right)^3 - 3\left(-\alpha_k\right)\left(-\frac{B}{\alpha_k}\right) + A\left(-\alpha_k\right) + 2B & = & -\alpha_k^3 - 3B - A\alpha_k + 2B\\
 & = & -\left(\alpha_k^3 + A\alpha_k + B\right)\\
 & = & 0
\end{array}.
\end{equation*}
Therefore $\mathcal{C}_E$ is incident to each point $\phi(P_i, P_j)$.
\end{proof}

The lines in the $uv$-plane passing incident to the point $\phi(O,O) = [0:1:0]$ at infinity are parametrized by $u = \lambda$. In this plane, (\ref{thecubic}) becomes
\[\mathcal{C}_E : u^3 - 3uv + Au +2B = 0.\]
Thus the points on $\mathcal{C}_E$ in the $uv$-plane are
\begin{equation}
\label{cubicparam}
(u(\lambda), v(\lambda)) = [u(\lambda) : v(\lambda) : 1] = \left[\lambda : \frac{\lambda^3 + A\lambda + 2B}{3\lambda} : 1\right],
\end{equation}
with $\lambda = 0$ corresponding to the point $[0:1:0]$ and infinity.

Write
\[\Omega_{\lambda} := \psi(u(\lambda),v(\lambda))(u(\lambda)^2 - 4v(\lambda)).\]
Explicitly,
\begin{equation}
\label{cubicOmega}
\Omega_{\lambda} = -\frac{\left(\lambda^3 + A\lambda - B\right)^2\left(\lambda^3 + 4A\lambda + 8B\right)^2}{81\lambda^4}.
\end{equation}
By (\ref{omegasquared}), the point $(u(\lambda), v(\lambda))$ lifts to rational points in $E^2/(\zmod{4})$ if and only if $\Omega_{\lambda} \in \left(K^*\right)^2$.
\begin{prop}
\label{Z4rational}
If $K$ contains $i$, then $\Omega_{\lambda}$ is a square in $K$ for every $\lambda \in K^*$. In this case, every rational point on $\mathcal{C}_E$ lifts to rational points on $E^2/(\zmod{4})$.
\end{prop}

\begin{proof}
This follows from (\ref{cubicOmega}), which shows that $\Omega_{\lambda} \in -1\cdot (K^*)^2$ when it is non-zero.
\end{proof}

\section{Irreducible points}

In this section, I show that infinitely many of the rational points produced in the previous section are irreducible. The argument closely follows that used in the irreducibility step of the proof of Proposition \ref{quadratic}. The primary difference is that I will make use of the more general Weil bounds in counting the number of irreducible points on the reductions of $E$ over large primes. The field $K$ will henceforth be as in the statement of Theorem \ref{main}.

Suppose that $\phi(P,Q) = [u:v:q] \in \mathbb{P}^2 \equiv E^2/D_4$. From (\ref{D4quotientmap}), the Weierstrass coordinates $x(P)$ and $x(Q)$ satisfy the equation
\[q_{u,v}(x) := x^2 - ux + v.\]
Thus the Galois action on the $\zmod{4}$-orbit of $(P,Q)$ is transitive if and only if $q_{u,v}(x)$ is irreducible.

Write
\[q_{\lambda}(x) := 3\lambda x^2 - 3\lambda^2 x + (\lambda^3 + A\lambda + 2B),\]
which has discriminant
\[D_{\lambda} = 6\lambda^4 - 12\lambda^3 -12A\lambda^2 -24B\lambda.\]
Let $P$ be a prime of $\ints{K}$ of large norm. By (\ref{cubicparam}), $q_{u(\lambda),v(\lambda)}(x)$ factors over the residue field $k_P$ if and only if $D_{\lambda}$ is a quadratic residue in $k_P$. The proportion of parameters $\lambda \in K$ for which $D_{\lambda}$ is a quadratic residue in $k_P$ may be estimated by counting the $k_P$-points on the curve
\[C : \mu^2 - D_{\lambda} = 0.\]
As $D_{\lambda}$ is not a square in $k_P[\lambda]$, $C$ is irreducible of genus 1. By the Weil bound for $C$, there are roughly $|k_P|/2$ values of $\lambda$ in $k_P$ for which $D_{\lambda}$ is a quadratic residue in $k_P$. Therefore, for roughly half of the parameters $\lambda$, $q_{u(\lambda), v(\lambda)}(x)$ is irreducible over $k_P$. Therefore, given distinct large primes $P_1, \ldots, P_n$ and $Q_1, \ldots, Q_m$ in $\ints{K}$, one may find using the Chinese remainder theorem parameters $\lambda$ for which $q_{u(\lambda), v(\lambda)}(x)$ is irreducible over the residue fields $k_{P_i}$ but factors over each field $k_{Q_j}$. Thus the ramification behaviour of the extension of $K$ generated by $q_{u(\lambda), v(\lambda)}(x)$ over larger and larger lists of large primes can be arbitrarily specified, proving that there are infinitely many linearly disjoint $\zmod{4}$-extensions $L$ of $K$ with points $P$ on $E$ defined only over $L$.

The nature of the action (\ref{D4act}) is such that any point $P$ defined only over an extension $L$ of $K$ as above necessarily satisfies the hypothesis of Lemma \ref{mereltrick}. Therefore, by Merel's theorem, there are infinitely many $\zmod{4}$-extensions $L$ of $K$ over which $E$ gains rank.

For a point $(P,Q) \in E^2$, let $l$ denote the line joining $P$ and $Q$ in $\mathbb{P}^2$. By (\ref{D3act}), the Galois action on the $\zmod{3}$-orbit of $(P,Q)$ is transitive if and only if the polynomial $p_l(x)$ defined in (\ref{pl}) is irreducible.

For the lines $l_{\lambda}$ given by (\ref{quarticparam}),
\[p_{\lambda}(x) = x^3 - \frac{1}{36\lambda^6}\left(6\lambda^2x + 3A\lambda^4 - 1\right)^2 + Ax + B.\]
Let $P$ be a large prime of $\ints{K}$. Clearing denominators and applying the cubic formula, one sees that $p_{\lambda}(x)$ is irreducible if and only if
\[D_{\lambda}' := \frac{27A^2\lambda^8 - 108B\lambda^6 - 18A\lambda^4 - 1}{4} \textrm{ is a cube in } k_P.\]
As before, the proportion of $\lambda \in K$ for which this condition holds may be estimated by counting $k_P$-points on the curve
\[C' : \mu^3 - D_{\lambda}' = 0.\]
A singular point on $C'$ must satisfy $\mu = 0$, and thus the corresponding value of $\lambda$ must be a multiple root of $D_{\lambda}'$. A multiple root of $D_{\lambda}'$ corresponds to a line in $\mathbb{P}^2$ which intersects $E$ at a single point with multiplicity greater than 3, which is impossible since $E$ is smooth. Therefore $C'$ itself is a smooth curve. As $K$ contains its cube roots of unity, there are by the Weil bound for $C'$ roughly $|k_P|/3$ values of $\lambda$ in $k_P$ for which $D_{\lambda}'$ is a cube in $k_P$. This allows us to apply the same argument as in the previous case to deduce that there are infinitely many $\zmod{3}$-extensions of $K$ over which $E$ gains rank.

Substituting $\omega = \lambda^2$ in $D_{\lambda}'$ yields the expression
\[D_{\omega}' := \frac{27A^2\omega^4 - 108B\omega^3 - 18A\omega^2 - 1}{4}.\]
As with $C'$, there are roughly $|k_P|$ $k_P$-points on the curve
\[C'' : \mu^3 - D_{\omega}' = 0.\]
Half of these points -- the ones for which $\omega$ is a square in $k_P$ -- correspond to the $k_P$-points on $C'$. Therefore, there are roughly $5|k_P|/6$ values of $\omega$ which are not quadratic residues in $k_P$ such that the polynomials $p_{\sqrt{\omega}}(x)$ are irreducible in $k_P[x]$. The lifts to $E^2/(\zmod{6})$ of the images of the corresponding points $\hat{l}_{\pm\sqrt{\omega}}$ in $\weightproj \equiv E^2/D_6$ are irreducible. This is because the irreducibility of $p_{\sqrt{\omega}}(x)$ implies that the action of $r_6^2 = r_3$ on the corresponding $\zmod{6}$-orbits is induced by the action of $\absgal{K}$. Finally, the Chinese remainder theorem argument allows us to deduce that there are infinitely many $\zmod{6}$-extensions of $K$ over which $E$ gains rank.

This concludes my proof of Theorem \ref{main}.

\bibliography{gr}{}
\bibliographystyle{alpha}{}

\end{document}